\documentclass[conference,10pt]{IEEEtran}
\usepackage{array,url}
\usepackage[margin=1.2in]{geometry}
\usepackage[geometry]{ifsym}
\usepackage{amssymb, amsfonts, amsmath, amsthm, amscd, mathrsfs, mathtools, color}
\usepackage{todonotes, comment, verbatim, fullpage}
\usepackage{graphicx, float, epsfig, enumerate, fixmath, bbm, latexsym, cite, scrextend}
\usepackage{pgfplots}
\pgfplotsset{compat=1.12}
\usepackage{subcaption}

\usepackage[hidelinks]{hyperref}
\usepackage{xcolor}
\hypersetup{
    colorlinks,
    linkcolor={blue!100!black},
    citecolor={green!50!black},
    urlcolor={red!80!black}
}

\newcommand{\rmd}{\mathrm{d}}
\newcommand{\bbE}{\mathbb{E}}\newcommand{\rme}{\mathrm{e}}

\newcommand{\bbR}{\mathbb{R}}

\newcommand{\sfA}{\mathsf{A}}
\newcommand{\sfB}{\mathsf{B}}

\newcommand{\sfI}{\mathsf{I}}

\newcommand{\cB}{\mathcal{B}}

\newcommand{\cF}{\mathcal{F}}

\newcommand{\cN}{\mathcal{N}}

\newcommand{\sign}{\mathsf{sign}}

\newtheorem{thm}{Theorem}

\newtheorem{lem}[thm]{Lemma}
\newtheorem{prop}[thm]{Proposition}

\begin{document}

\title{Multivariate Priors and the Linearity of Optimal Bayesian Estimators under Gaussian Noise}

%
%
%
 \author{%
  \IEEEauthorblockN{Leighton P. Barnes\IEEEauthorrefmark{1},
                    Alex Dytso\IEEEauthorrefmark{2}, Jingbo Liu\IEEEauthorrefmark{3},
                    and H. Vincent Poor\IEEEauthorrefmark{4}}
  \IEEEauthorblockA{\IEEEauthorrefmark{1}%
                    Center for Communications Research,
                    Princeton, NJ 08540, USA,
                    l.barnes@idaccr.org}
   \IEEEauthorblockA{\IEEEauthorrefmark{2}%
                    Qualcomm Flarion Technology, Inc.,
                    Bridgewater, NJ 08807, USA,
                    odytso2@gmail.com}
     \IEEEauthorblockA{\IEEEauthorrefmark{3}%
  University of Illinois, Urbana-Champaign, IL 61820, USA, jingbol@illinois.edu
    }
  \IEEEauthorblockA{\IEEEauthorrefmark{4}%
                    Princeton University, 
                    Princeton, NJ 08544,  USA,
                    poor@princeton.edu}
}

\maketitle


\begin{abstract}
Consider the task of estimating a random vector $X$ from noisy observations $Y = X + Z$, where $Z$ is a standard normal vector, under the $L^p$ fidelity criterion. This work establishes that, for $1 \leq p \leq 2$, the optimal Bayesian estimator is linear and positive definite if and only if the prior distribution on $X$ is a (non-degenerate) multivariate Gaussian. Furthermore, for $p > 2$, it is demonstrated that there are infinitely many priors that can induce such an estimator.
\end{abstract}

\section{Introduction}

Consider a random vector \(X \in \mathbb{R}^n\) that is observed through noisy observation \(Y \in \mathbb{R}^n\) with
\begin{equation}
    Y = X + Z
\end{equation}
and \(Z \in \mathbb{R}^n\) standard normal and independent of \(X\). The optimal Bayesian estimator for estimating \(X\) from \(Y\) with different $L^p$ losses can be described as follows: for \(p, k \ge 1\)
\begin{equation}
    \inf_{f: \, f \text{ measurable}} \mathbb{E} \left[ \ell_{p, k} \left(X - f(Y)\right) \right]\label{eq:gen_estimation}
\end{equation}
where, for \(x \in \mathbb{R}^n\), the loss function is defined as
\begin{equation}
    \ell_{p, k}(x) = \| x \|_k^p,  \text{ with } \| \cdot \|_k = \left(\sum_{i=1}^n |x_i|^k\right)^\frac{1}{k}.
\end{equation}
For the case of \(n=1\), without loss of generality, we can assume that \(k=1\).
Finding the optimal estimator in \eqref{eq:gen_estimation} for all combinations of \(p, k\) is a difficult task, and closed-form expressions are  known only in some special cases. For example, for \(p = k = 2\), the optimal estimator is given by the conditional mean \cite{poor1998introduction}. For the case of \(p  = 1\), the optimal estimator is known as the spatial median  \cite{small1990survey,milasevic1987uniqueness} and  has a closed-form expression only for  \(n = 1\). In the vector case, it is known that for every $n,k$ and $p$ the Gaussian distribution on $X$ induces a linear estimator of the form $ f(Y)= \sfA Y$ where $\sfA \in \bbR^{n \times n}$ is a positive semidefinite matrix.\footnote{For completeness this claim is shown in Proposition~\ref{prop:Gauss_sol} (Section~\ref{sec:orthog_like_cond}). }

In this work, building on our previous work in \cite{barnes20231}, which is concerned with the scalar case, we are interested in studying the converse of this statement. Formally, suppose that \(f_{p,k}(y)\) is the optimal estimator in \eqref{eq:gen_estimation};\footnote{Due to the convexity of \(f \mapsto \mathbb{E}[\ell_{p,k}(X - f(Y))]\), such an estimator exists and is unique for \(p,k>1\) \cite[Prop.~1]{dytso2017minimum}.} our goal is to find the set of priors on the input \(X\) such that for all\footnote{All equalities  in the paper are understood to hold in the almost sure sense.} \(y \in \mathbb{R}^n\)
\begin{equation}
    f_{p,k}(y) = \mathsf{A}y \label{eq:Induces_linarity}
\end{equation}
for some positive semidefinite matrix \(\mathsf{A} \in \mathbb{R}^{n \times n}\).

\subsection{Past Work }
In the scalar case, in \cite{barnes20231} the authors of this paper demonstrated that for  \(1 \le p \le 2\), only a Gaussian prior induces linearity of the optimal estimator, and for \(p>2\), there are infinitely many priors that induce linearity.

For the $L^2$ (i.e., $p,k=2$) error scenario, the identification of the set of distributions on the variable $X$ that would result in a linear conditional mean has been well-established. Specifically, in the case of Gaussian noise,  the sole prior inducing linearity is a Gaussian distribution with zero mean and with  covariance given by  $(\sfI -\sfA)^{-1} \sfA$, denoted as $X \sim \mathcal{N}(0,(\sfI -\sfA)^{-1} \sfA)$; see for example \cite[App.~B]{barnes20231} for four different methods of showing this.   

In the $L^2$ case, this understanding extends beyond Gaussian noise scenarios. For instance, when the conditional distribution $P_{Y|X}$ belongs to an \emph{exponential family}, it is established that the conditional mean is linear if and only if $X$ follows a \emph{conjugate prior} \cite{diaconis1979conjugate,chou2001characterization}. In such cases, the matrix $\sfA$ is determined by  ${\rm Cov}(X,Y) {\rm Cov}^{-1}(Y)$. Conjugate priors find application in modeling various phenomena in statistical and machine learning contexts \cite{bishop2006pattern}.

Addressing additive noise channels, where $Y=X+N$ and $N$ is not necessarily Gaussian, in the scalar case, \cite{akyol2012conditions} outlined necessary and sufficient conditions for linearity in optimal Bayesian estimators for $L^p$ Bayesian risks, where $p$ is restricted to \emph{even} values. Specifically, the authors derived the characteristic function of $X$ in terms of the characteristic function of $N$. However, these findings do not extend to the case $p=1$. The authors of \cite{akyol2012conditions} also derived similar conditions for the vector case but only for the case of $p=k=2$.

In the $L^2$ scenario, alongside uniqueness results, stability outcomes have also been established. Particularly, for Gaussian and Poisson noise models, if the conditional expectation closely approximates a linear function in the $L^2$ distance, the distribution of $X$ must be in proximity to a corresponding prior distribution (Gaussian for Gaussian noise and gamma for Poisson noise) in the \emph{L\'evy distance} \cite{du2018strong,dytso2020estimation}.

\subsection{Contributions and Outline }
The remaining part of this section is dedicated to notation. Section~\ref{sec:Preliminaries} presents preliminary results. For example, in Proposition~\ref{prop:conv_ident} we map our problem to a convolution equation. Section~\ref{sec:main_result} is dedicated to our main theorem. In particular, Theorem~\ref{thm:main_thm} shows that for $1 \le p=k\le 2$, Gaussian is the only prior that induces a linear and positive definite estimator. In Section~\ref{sec:non_trivial_sol} we show that for $p>2$, there are
infinitely many prior distributions that induce such an estimator. 
\subsection{Notation}

We will use the following multi-index notation. For a multi-index $\alpha =(\alpha_1, \ldots, \alpha_d$), we let $|\alpha| = \alpha_1+\ldots + \alpha_d$ and for  $f :\bbR^d \to \bbR$ we define 
\begin{equation}
    \partial_x^\alpha f(x) = \left( \frac{\partial}{ \partial x_1} \right)^{\alpha_1}  \ldots  \left( \frac{\partial}{ \partial x_d} \right)^{\alpha_d} f(x) 
\end{equation}
and  $\alpha! =\alpha_1! \dots \alpha_d!$.    The probability density of a standard scalar Gaussian random variable is denoted by $\phi_0$, and the probability density of an $n$-dimensional standard Gaussian random vector is denoted by $\phi$. 

\section{Preliminaries}
\label{sec:Preliminaries}
In this section, we present some of the required preliminaries. 

\subsection{Orthogonality Like Condition} \label{sec:orthog_like_cond}

The key to most of our proofs will be the following \emph{ orthogonality principle}  like result, which provides a sufficient and necessary condition on $X$ for optimality of the linear estimator \cite{akyol2012conditions}. 

\begin{thm}
   For $p,k \ge 1$, $X$  satisfies \eqref{eq:Induces_linarity} iff for all $y \in \bbR^n$
    \begin{equation}
        \bbE \left[ \ell'_{p,k}  \left(X-  \sfA y\right)  \phi(y-X) \right] = 0_n  \label{eq:Orthog_cond}
\end{equation}
where
\begin{align}
\ell'_{p,k}(x) &= \nabla_x   \| x \|_k^p  = p  \| x \|_k^{p-1}  \nabla _x\| x \|_k  \\
&=  p      \frac{\| x \|_k^{p-1}}{ \| x\|_k^{k-1} }  \left [ \begin{array}{c}     |x_1|^{k-1} \sign(x_1) \\ . \\ .  \\     |x_n|^{k-1} \sign(x_n)    \end{array}  \right ]. 
\end{align} 
\end{thm}

We now show that a Gaussian prior on $X$ satisfies \eqref{eq:Orthog_cond}. 

\begin{prop}\label{prop:Gauss_sol} Fix some $p,k \ge 1$ and $0 \preceq \sfA \prec \sfI$. Then,  $X \in \mathcal{N}( 0_n, (\sfI -\sfA)^{-1} \sfA)$ satisfies \eqref{eq:Orthog_cond}. 
\end{prop} 
\begin{proof}  First, note that 
 $X|Y =y  \sim \mathcal{N} (\sfA y;  \sfA)$, which implies that the conditional expectation can be written as an unconditional one 
\begin{align}
\bbE \left[   \ell'_{p,k} (X-  \sfA Y)  \mid  Y =y  \right] & =  \bbE \left[   \ell'_{p,k} ( \sfA Z +\sfA y -  \sfA y)    \right]  \\
&=  \bbE \left[   \ell'_{p,k} ( \sfA Z)   \right] .
\end{align} 
Next, since $- Z  \stackrel{d}{=} Z$, we have that 
\begin{equation}
\bbE[   \ell'_{p,k} ( \sfA Z)   ]  =  \bbE[   \ell'_{p,k} ( - \sfA Z)   ]  =  -\bbE[   \ell'_{p,k} ( \sfA Z)   ] 
\end{equation} 
which implies that $\bbE[   \ell'_{p,k} ( \sfA Z)   ]  =0_n$.  
\end{proof}

\subsection{An Equivalent Convolution Identity} 
To show the converse to Proposition~\ref{prop:Gauss_sol}, we re-write the condition in \eqref{eq:Orthog_cond} in terms of a convolution. Let $P_X$ be the probability distribution of $X$.
\begin{prop} \label{prop:conv_ident} Fix some $p,k \ge 1$ and $0 \prec \sfA$. $P_X$ satisfies \eqref{eq:Orthog_cond} iff for all $y \in \bbR^n$
\begin{equation} \label{eq:conv}
0_n  = \int  \tilde{\ell}'_{p,k} ( x-  y)  \phi(y-x) \rmd \mu( x), 
\end{equation} 
where\footnote{For a symmetric and positive semidefinite matrix $\sfA$, the matrix square root is the matrix such that $\sfA = \sfA^{\frac12 }\sfA^{\frac12 }$. Such a matrix always exists, is unique, and is positive semidefinite \cite{horn2012matrix}.  }
\begin{align}
\tilde{\ell}'_{p,k} (x)  &=  \ell'_{p,k} \left(   \sfA^{ \frac{1}{2}} x \right),\\
\rmd \mu( x)&= \rme^{   \frac{ x^T(\sfI -\sfA) x }{ 2} } \,  \rmd  P_ {  \sfA^{ - \frac{1}{2}}  X} ( x) .\label{eq:definition_mu}
\end{align} 
\end{prop}
\begin{proof}
For all $y \in \bbR^n$, we have that 
\begin{align}
0_n  &= \int  \ell'_{p,k} (x-  \sfA y)  \phi(y-x) \, \rmd  P_X(x) \\
\Leftrightarrow  0_n  &= \int  \ell'_{p,k} \left(   \sfA^{ \frac{1}{2}}  (  \sfA^{ -\frac{1}{2}}  x-    \sfA^{ \frac{1}{2}}  y) \right)  \phi(y-  x) \, \rmd P_ {X} ( x) 
\end{align}
where $\sfA^{\frac12}$ denotes the matrix square root, which is well defined since $\sfA$ is symmetric and positive definite.  

Next, by doing a change of variable on $X$, we note that for all $y \in \bbR^n$
\begin{align} 
  0_n  &= \int  \tilde{\ell}'_{p,k} ( \sfA^{ -\frac{1}{2}}  x-  \sfA^{ \frac{1}{2}} y)  \phi(  y-  x) \, \rmd P_ {X} ( x) \\ 
\Leftrightarrow  0_n  &= \int  \tilde{\ell}'_{p,k} ( x-  \sfA^{ \frac{1}{2}} y)  \phi(  y-   \sfA^{ \frac{1}{2}}  x)  \, \rmd P_ {  \sfA^{ - \frac{1}{2}}  X} ( x) \; .
\end{align}
By performing a change of variable on $y$, we have that  
\begin{equation}
 0_n  = \int  \tilde{\ell}'_{p,k} ( x-  y)  \phi(  \sfA^{ -\frac{1}{2}} y-   \sfA^{ \frac{1}{2}}  x) \, \rmd P_ {  \sfA^{ - \frac{1}{2}}  X} ( x) \; .
\end{equation}
Finally, manipulating the exponent, we arrive at 
\begin{align}
 0_n  &= \int  \tilde{\ell}'_{p,k} ( x-  y)  \rme^{ y^T x   - \frac{ \| \sfA^{ \frac{1}{2}}  x \|_2^2 }{ 2} }   \, \rmd P_ {  \sfA^{ - \frac{1}{2}}  X} ( x) \\  
\Leftrightarrow  0_n  &= \int  \tilde{\ell}'_{p,k} ( x-  y)  \rme^{ y^T x   - \frac{ \|   x \|_2^2 + x^T(\sfI -\sfA) x}{ 2}  }   \rmd P_ {  \sfA^{ - \frac{1}{2}}  X} ( x) \\ 
\Leftrightarrow  0_n  &= \int  \tilde{\ell}'_{p,k} ( x-  y)  \phi(y-x) \, \rmd \mu( x) \; . 
\end{align} 
\end{proof} 

\subsection{Tempered Distributions and Other Ancillary Results} 
Our main technique uses the theory of distributions and tempered
distributions from functional analysis, the background of which can be found in \cite{stein2011functional}. 
We now summarize some of the required results.  

\begin{lem} \label{lem:summary:Properties}
\text { }
\begin{itemize}

\item \cite[p.~293]{folland1999real}: Suppose that there exists a constant $C>0$ and an integer $N>0$ such that for all Schwartz-class functions $\varphi$,
\begin{equation}
\left| \int \varphi(x) \rmd \mu(x) \right| \le C\sum_{|\alpha|<N,|\beta|<N}\sup_x |x^\alpha\partial_x^\beta\varphi|\; .
\end{equation}
Then, $\mu$ is a tempered distribution.

\item \cite[Thm.~1.7]{stein2011functional}: If $\nu$ is a tempered distribution supported at the origin, then the Fourier transform $ \widehat{\nu} $ of $\nu$ is given by 
\begin{equation}
\widehat{\nu} (x)  = \sum_{|\alpha| \le N} c_\alpha x^\alpha,
\end{equation}
for some $N<\infty$.
\end{itemize}

\end{lem}

\begin{lem}\label{lem:support_of_FT} Let 
\begin{equation}
  \psi_i(x) =   | x_i|^{k-1}\sign(x_i)\phi_0(x_i)\prod_{j\neq i} \phi_0(x_j)
\end{equation}
   Then, for $k \in [1,2]$, the Fourier transform 
    $\widehat\psi_i(\omega) = 0$ if and only if $\omega_i = 0$.
\end{lem}
\begin{proof}
    Note that by linearity we have that
    \begin{equation}
    \widehat\psi_i (\omega)  \propto \cF \left\{ | \cdot|^{k-1}\sign(\cdot)\phi_0(\cdot) \right\} (\omega_i)  \prod_{j \neq i} \phi_0(\omega_j)  .
    \end{equation}
    Clearly,  $\prod_{j \neq i} \phi_0(\omega_j) $  is always non-zero. Moreover, the fact that for $  k \in [1,2]$, the function $\omega_i \mapsto \cF \{ | \cdot|^{k-1}\sign(\cdot)\phi_0(\cdot) \} (\omega_i)$ has zeros only at the origin can be found in \cite[Thm.~14]{barnes20231}. 
\end{proof}

\begin{lem} \label{lem:monomial_convolution}
    Given $N<\infty$ and  the sum 
    \begin{equation}
        {\rm mpol}_N(x) = \sum_{|\alpha| \le N} c_\alpha x^\alpha,
    \end{equation}
    suppose that  for all $y\in \bbR^n$ and $i \in \{1,\ldots, n\}$
    \begin{equation}
        \bbE \left[  f(Z_i)    {\rm mpol}_N(Z -y) \right] = 0,
    \end{equation}
    where $f:\bbR \to \bbR$ is a nonzero odd  increasing function. 
    Then,  ${\rm mpol}_N$ is a constant function. 
\end{lem}
\begin{proof}
  
The proof follows by taking a derivative of maximum degree minus one and evaluating it at zero. That is, given an $i$ and a multi-index $\alpha$ such that $|\alpha| =N-1$
\begin{align}
0 
&=    \bbE \left[ f(Z_i)  \partial^{\alpha}_{y} {{\rm mpol}}_N(Z-y) \right]  \Big |_{y=0}\\ 
&= \bbE \left[ C_i f(Z_i) Z_i    \right] ,
\end{align}
where the constant $C_i$  is a nonzero multiple of $c_{\alpha+e_i}$, $e_i$ being the $i$-th standard basis vector.
Since $\bbE \left[ f(Z_i) Z_i    \right]>0$, the above implies that $C_{i }=c_{\alpha+e_i} =0$ for all $i$ and $|\alpha|=N-1$.  Therefore, $ {{\rm mpol}}_N(x)$ is a polynomial of maximum degree $N-1$. Now repeating this procedure $N-1$ times, we arrive at $ {{\rm mpol}}_N(x)=c_0$.
\end{proof}

\section{Main Result} \label{sec:main_result}

The main result of this paper is the following theorem. 
\begin{thm} \label{thm:main_thm}
   Fix some $0 \prec \sfA \prec \sfI$ and assume that $1 \le p =k \le 2$. Then,  $f_{p,k}(y) = \mathsf{A}y, \forall y\in \bbR^n$ if and only if $X \sim \cN \left( 0, (\sfI -\sfA)^{-1} \sfA \right)$. If $\sfA$ has an eigenvalue greater than or equal to one, then $f_{p,k}(y) = \mathsf{A}y$ is an inadmissible estimator. 
\end{thm}

The proof is given next in a series of steps.

\subsection{$\mu$ is a Tempered Distribution}  
Let $B$ be an orthogonal matrix such that the first row of $\lambda B$ is the same as the first row of $A^\frac{1}{2}$. The constant $\lambda$ is just a normalization factor that is equal to the magnitude of the first row of $A^\frac{1}{2}$. We define the following function:
\begin{align}
G(u) & = -\int_{-\infty}^{u_1} |w|^{k-1}\mathsf{sign}(w)\phi_0(w)\prod_{i\neq 1}\phi_0(u_i) \rmd {w} 
\end{align}
and note that if $H(u) = \frac{\partial}{\partial u_1} G(u)$, then 
\begin{equation}
H( B t) = -\frac{1}{|\lambda|^{k-1}}(\tilde\ell'_{p,k})_1(t)\phi(t) \; .    
\end{equation}

The functions $H$ and $G$ decay sufficiently fast in the sense that
\begin{equation}
    \int \max \{ |H(x)|  , G(x) \}  \rmd \mu(x) <\infty;
\end{equation}
so by the dominated convergence theorem we have that \eqref{eq:conv} implies, for some constant $C$, 
\begin{equation} 
C  = \int  G(y-x) \rmd\mu( x), \, \forall y.
\end{equation}

Let $\cB(r)$ be the Euclidean ball centered at zero with radius $r$, and let $\mathcal{R}(r)$ be the cube centered at zero with side-length $2r$. First, note that  
\begin{equation}
    G(t) \ge  c \,  1_{ \mathcal{R}(1) }(t)
\end{equation}
for some constant $c >0$. 
Second, since $\mu $ is a non-negative measure, it follows that 
\begin{align}
    C &= \int  G(y-x) \rmd\mu( x)\\
    & \ge  c \int  1_{\mathcal{R}(1)}(y-x) \rmd\mu( x)\\ 
    & \ge   c \mu( \mathcal{R}_y(1)) \; .
\end{align}
Next, by taking disjoint unions and summing the corresponding measures, there exists a constant $C_1$ such that 
\begin{equation} \label{eq:ball_meas_bound}
\mu( \cB(R)) \le \mu(\mathcal{R}(R)) \leq C_1 R^n
\end{equation}
for every $R>1$. 

Now for any Schwartz-class function $\varphi$,
\begin{align}
\bigg| \int & \varphi  (x) \rmd \mu( x) \bigg| \\
& \leq \int |\varphi(x)|\left(\frac{1}{1+\|x\|_2^2}\right)^M\left(1+\|x\|_2^2\right)^M \rmd \mu( x) \\
& \leq \int\left(\frac{1}{1+\|x\|_2^2}\right)^M\rmd \mu( x) \sup_x (1+\|x\|_2^2)^M|\varphi(x)| . \label{eq:seminorm}
\end{align}
The first term in \eqref{eq:seminorm} can be bounded as follows:
\begin{align}
\int & \left(\frac{1}{1+\|x\|_2^2}\right)^M\rmd \mu( x) \\
& = \int_0^\infty \mu\left(\left\{x\bigg|\left(\frac{1}{1+\|x\|_2^2}\right)^M > t\right\}\right)\rmd t \\
& = \int_0^1 \mu\left(\mathcal{B}\left(\sqrt{\frac{1}{t^{\frac{1}{M}}}-1}\right)\right) \rmd t \\
& \leq C_1 \int_0^1 \left(\frac{1}{t^{\frac{1}{M}}}-1\right)^\frac{n}{2} \rmd t \label{eq:ball_ineq_step}
\end{align}
where in \eqref{eq:ball_ineq_step} we have used \eqref{eq:ball_meas_bound}. Choosing $M=n$ we have
\begin{equation}
\int \left(\frac{1}{1+\|x\|_2^2}\right)^M\rmd \mu( x) \leq C_1\int_0^1 t^{-\frac{1}{2}} \rmd t = 2C_1 \; .
\end{equation}
Thus
\begin{equation}
\bigg| \int \varphi (x) \rmd \mu( x) \bigg|  \leq 2C_1\sup_x (1+\|x\|_2^2)^M|\varphi(x)| \; ,
\end{equation}
and by Lemma \ref{lem:summary:Properties} we have established that $\mu$ is a tempered distribution.

\subsection{The Prior $\mu$ is Constant for $p =k$ and $1\leq p \leq 2$}
The convolution identity in \eqref{eq:conv} has $n$ components, each of which yields a particular convolution equation. Consider the $i$th of these equations, i.e.,
\begin{equation} 
0 = \int (\tilde\ell'_{p,k})_i(x-y)\phi(x-y)\mu(dx) \; . \label{eq:one_relation} \end{equation}
Now  the convolutional kernel $(\tilde\ell'_{p,k})_i(x)\phi(x)$ depends only on the $i$th row of $\sfA^{\frac{1}{2}}$, and not on any other part of the matrix $\sfA^{\frac{1}{2}}$. We can therefore replace $\sfA^\frac{1}{2}$ with $\lambda_i \sfB_i$ for an orthogonal matrix $\sfB_i$ in the definition of $\tilde\ell'_{p,k}$ for the purposes of analyzing a single relation \eqref{eq:one_relation}. The $i$th row of $\sfB_i$ will be the normalized $i$th row of $\sfA^{\frac{1}{2}}$, and $\lambda_i$ will be the corresponding norm. The rest of the $\sfB_i$ matrix is any completion of the orthonormal basis. Define
\begin{align}
\psi_i(x) & = (\tilde\ell'_{p,k})_i(B_i^{-1}x)\phi(B_i^{-1}x) \\
& = (\tilde\ell'_{p,k})_i(B_i^{-1}x)\phi(x) \\
& = (\tilde\ell'_{p,k})_i(B_i^{-1}x)\phi_0(x_i)\prod_{j\neq i} \phi_0(x_j) \\
& = |\lambda_i x_i|^{k-1}\sign(x_i)\phi_0(x_i)\prod_{j\neq i} \phi_0(x_j) \; .
\end{align}
Taking the Fourier transform of $\psi_i(x)$ (Lemma~\ref{lem:support_of_FT}), we see that $\widehat\psi_i(\omega) = 0$ if and only if $w_i = 0$. Therefore the Fourier transform of the convolutional kernel
\begin{equation}
\psi_i(B_i x) = (\tilde\ell'_{p,k})_i(x)\phi(x)
\end{equation}
is zero at frequency $\omega$ if and only if
\begin{equation} \label{eq:zeros}
 (B_i\omega)_i = (A^{\frac{1}{2}}\omega)_i = 0 \; .
\end{equation}

Since $\mu$ is a tempered distribution, we can take its Fourier transform and effectively use the convolution theorem, from which it is clear from \eqref{eq:zeros} that \eqref{eq:one_relation} implies $\widehat\mu$ must be supported only on $\omega$ such that
$(A^{\frac{1}{2}}\omega)_i= 0$ (for a formal proof using the definition of the support of distributions, see \cite[Lemma~11]{barnes20231}). This being true for all $i=1,\ldots,n$, and $\sfA^{\frac{1}{2}}$ being invertible, we conclude that $\widehat\mu$ is supported only at the origin.

It is a standard result (see Lemma \ref{lem:summary:Properties}), that $\widehat\mu$ being supported only at the origin implies that $\mu$ can be represented as a finite sum of monomials: for $N<\infty$
\begin{equation}
\mu(x) = \sum_{ |\alpha|  \le N }c_{\alpha}x^\alpha \; .
\end{equation}

Now  \eqref{eq:one_relation} can be written as $0 = \bbE [|Z_i|^{k-1}\sign(Z_i) \mu(Z-y) ]$ and for all $i $. Noticing that  $z \mapsto |z|^{k-1} \sign(z)$ is an odd increasing function and using Lemma~\ref{lem:monomial_convolution}, we conclude that  $\mu = c_{0}=c$. Next, from the definition of $\mu$ in \eqref{eq:definition_mu}, we have that
\begin{equation}
 c  =    \rme^{   \frac{ x^T(\sfI -\sfA) x }{ 2} }  \,   \rmd P_ {  \sfA^{ - \frac{1}{2}}  X} ( x),
\end{equation}
which  implies that
\begin{equation}
      \rmd P_ {  X} (x) \propto \rme^{ -   \frac{ x^T \sfA^{ -\frac12   }(\sfI -\sfA) \sfA^{-\frac12} x }{ 2} }= \rme^{ -   \frac{ x^T \sfA^{-1}(\sfI -\sfA)  x }{ 2} }, \label{eq:final_result}
\end{equation}
where the last step is due to the symmetry of $\sfA$. Finally, note that $\rmd P_X $ in \eqref{eq:final_result} is a proper distribution if $\sfA \prec \sfI$, and if $\sfA$ has eigenvalue greater than or equal to one, then  $\rmd P_X $ is not a proper distribution. 

\subsection{Nontrivial Solutions for $p>2$}
\label{sec:non_trivial_sol}

Next, somewhat surprisingly, we show that for $p>2$ there are infinitely many priors that induce linearity.  For simplicity, we focus on the case of $n=1$, while the more general case of $n>1$ can be addressed by a straightforward extension of these ideas. 
\begin{thm}
    \label{thm13}
   Fix a $ p \in (2,\infty)$ and $0<a<1$. Then for every $|\rho| \le 1$ and $\theta \in \bbR$, there exists an $\omega$ such that the density
   \begin{equation}
f_X(x) \propto \rme^{- \frac{1-a}{a} \frac{x^2}{2}} \left(1 + \rho \cos \left( \frac{\omega x}{\sqrt{a}} +\theta \right) \right) \label{eq:lp_densities}
\end{equation}
induces a linear minimum $L^p$ estimator. Moreover,  for even $p$, 
 $\omega$'s are given  by the zeros of the probabilist's Hermite polynomial $H_{e_{p-1}}$. 
\end{thm}
\begin{proof} 
We show that there is an appropriate choice of $\omega$ such that the density in \eqref{eq:lp_densities} satisfies \eqref{eq:Orthog_cond} , which would imply that the above density induces linearity of the conditional $L^p$ estimator.  We have that 
\begin{align}
&\int_{-\infty}^\infty \ell_p'(x-y) \phi(y-x) \exp\left( (1-a) \frac{x^2}{2} \right) f_X(\sqrt{a}x)  \rmd x \notag\\
& =  \int_{-\infty}^\infty \ell_p'(x-y) \phi(y-x) \left(1 +\rho\cos \left( \omega x +\theta\right) \right)  \rmd x\\
& = \rho \int_{-\infty}^\infty \ell_p'(x-y) \phi(y-x)  \cos \left( \omega x +\theta \right)   \rmd x \label{eq:using_that_p_is_even} \\
& = \rho \mathsf{Re} \left \{ \rme^{-j \omega y +j \theta} \mathcal{F} \left( \sign(\cdot) | \cdot  |^{p-1} \phi(\cdot)  \right)(\omega)\right \} 
\end{align}
where  \eqref{eq:using_that_p_is_even} follows from the fact that the function is odd.

For even $p$, the proof is simple and 
\begin{align}
&\mathsf{Re} \left \{ \rme^{-j \omega y +j \theta} \mathcal{F} \left(  ( \cdot  )^{p-1} \phi(\cdot)  \right)(\omega)\right \} \notag\\ 
& = \mathsf{Re} \left \{ \rme^{-j \omega y  +j \theta } (-1)^{p-1} H_{e_{p-1}}(\omega)\phi(\omega)\right \}, \label{eq:using_heremite}
\end{align}
where \eqref{eq:using_heremite} follows by using the identity between the derivative of the Gaussian density and the probabilist's Hermite polynomials $H_{e_k}$. Note that $H_{e_{p-1}}$ has exactly $p-1$ zeros; thus placing $\omega$  at any of these locations will result in \eqref{eq:using_heremite} being equal to zero. The proof for all $p>2$ can be found in \cite[App.~D]{barnes20231}.
\end{proof}

Fig.~\ref{fig:example_lp_dens} shows  examples of the distributions in \eqref{eq:lp_densities} for $p=4$ where we note that  $H_{e_3}$ has zeros at $ \{ \pm \sqrt{3}, 0 \}$. 

\begin{figure}[h!]  
	\centering
\input{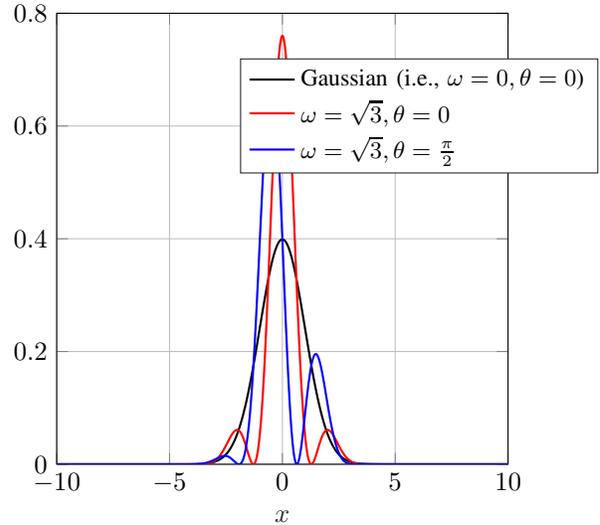}
	\caption{Example of probability densities in \eqref{eq:lp_densities} for $p=4$ and $\rho=1$.}
  \label{fig:example_lp_dens}
\end{figure}

\bibliographystyle{IEEEtran}
\bibliography{refs.bib}
\end{document}